\documentclass[11pt]{article}
\usepackage{a4, amsthm, amsmath, amssymb,amsbsy, epic,graphicx,mathrsfs,enumerate}
\usepackage[all]{xy}
\usepackage{multicol,longtable}

\newtheorem{thm}{Theorem}[section]

\newtheorem{cor}[thm]{Corollary}
\newtheorem{prop}[thm]{Proposition}
\newtheorem{defi}[thm]{Definition}

\theoremstyle{definition}
\newtheorem*{remark}{Remark}
\newtheorem{exam}{Example}[section]

\newcommand{\ra}{\rightarrow}
\newcommand{\mZ}{\mathbb{Z}}

\newcommand{\mQ}{\mathbb{Q}}
\newcommand{\mC}{\mathbb{C}}
\newcommand{\mR}{\mathbb{R}}
\newcommand{\mN}{\mathbb{N}}

\newcommand{\vtl}{\vartriangleleft}

\newcommand{\tzeta}{\tilde{\zeta}}
\newcommand{\loc}{\text{loc}}
\newcommand{\calM}{\mathcal{M}}
\newcommand{\bL}{\mathbf{L}}
\newcommand{\Gr}{\operatorname{Gr}}
\newcommand{\calX}{\mathcal{X}}
\renewcommand{\rm}{\mathrm}
\newcommand{\codim}{\operatorname{codim}}
\newcommand{\mbf}{\mathbf}
\newcommand{\Cal}{\mathcal}

\newcommand{\scr}{\mathscr}

\newcommand{\Gal}{\operatorname{Gal}}

\title{Reduced zeta functions of Lie algebras}
\author{Anton Evseev}
\date{}

\begin{document}
\maketitle

\begin{abstract}
We define reduced zeta functions of Lie algebras, which can be derived, via the Euler characteristic, from motivic zeta functions counting subalgebras and ideals.
We show that reduced zeta functions of Lie algebras possessing a suitably well-behaved basis are easy to analyse. 
We prove that reduced zeta functions are multiplicative under certain conditions and investigate which reduced zeta functions have functional equations. 
\end{abstract}

\section{Introduction}\label{reducedover}

Let $p$ be a prime. Let $L_p$ be a torsion-free finite dimensional Lie algebra over the ring $\mZ_p$ of $p$-adic integers. The \emph{zeta functions} of $L_p$ are defined by
\begin{eqnarray*}
\zeta_{L_p}^{\vtl}(s)& = & \sum_{n=0}^{\infty} a_{n,p}^{\vtl} p^{-ns}, \\
\zeta_{L_p}^{\le}(s)& = & \sum_{n=0}^{\infty} a_{n,p}^{\le} p^{-ns}
\end{eqnarray*}
where $a_{n,p}^{\vtl}$ (respectively $a_{n,p}^{\le}$) is the number of ideals (respectively subalgebras) of index $p^n$ in $L_p$. If $L$ is a Lie algebra over $\mZ$ additively isomorphic to $\mZ^d$, the \emph{local zeta functions} of $L$ are defined by 
$$\zeta_{L,p}^{*}(s)=\zeta_{L\otimes \mZ_p}^*(s)$$
(here, and in what follows, $*$ stands  either for $\vtl$ or for $\le$).

Properties of zeta functions of Lie algebras and of the related (via the Mal'cev correspondence) zeta functions of finitely generated nilpotent groups were first investigated in the seminal paper~\cite{GSS}. These functions have received considerable attention since its publication. Surveys of the topic are given in~\cite{survey, duSW}.
In this paper we shall concentrate on Lie algebras. However, using the Mal'cev correspondence, one can interpret the results of this paper in terms of zeta functions of nilpotent groups.

It will be more convenient for us to view $p^{-s}$ as a single variable. We modify the definitions accordingly:
$$ \tzeta_{L_p}^{*}(T):= \sum_{n=0}^{\infty} a_{n,p}^{*} T^n, $$ 
$$\tzeta_{L,p}^{*}(T):=\tzeta_{L\otimes \mZ_p}^{*}(T).$$ 
These are power series in one variable $T$. 
By \cite{GSS}, Theorem 3.5, they are rational functions in $T$.  

In many cases, $L$ has a \emph{uniform} ideal (or subalgebra) zeta function. That is,
$$\tzeta_{L,p}^{*}(T)=f(p,T)$$
 for almost all (that is, all but finitely many) values of $p$, where $f(X,Y)$ is a rational function in two variables. 
This paper considers, loosely speaking, the `reduced' zeta function 
$$R_L^* (T):=f(1,T).$$
We shall see that in certain cases, $R_L^*(T)$ can be easily calculated and has a much simpler form than the usual zeta function $\zeta_{L,p}^{*}(s)$. A reduced zeta function may be defined even when $\tzeta_{L,p}^{*}(T)$ is not uniform. For this purpose we use motivic zeta functions defined by M. du Sautoy and F. Loeser~\cite{duSL}. 

Section~\ref{motivic} contains definitions and properties of motivic zeta functions as well as those of the Euler characteristic. In Section \ref{main} we define reduced zeta functions in general and show that, unlike the usual zeta functions, many reduced zeta functions are relatively easy to analyse and to calculate. We also prove that, under suitable conditions,
$$R_{L\oplus N}^* (T)=R_L^* (T) R_N^* (T).$$
In Section~\ref{nice} we calculate reduced zeta functions for some Lie algebras possessing a well-behaved basis. 
In Section~\ref{funeqs} we investigate which reduced zeta functions have functional equations. In particular, we show that 
$$R_L^{\le} (T^{-1})= (-1)^d T^d R_L^{\le} (T),$$
provided the Lie algebra $L$ has a suitably well-behaved basis. We also find sufficient conditions for $R_L^{\vtl}(T)$ to have a functional equation. This result may help investigate functional equations of local ideal zeta functions. 

There is a price to be paid for these properties: a reduced zeta function encodes considerably less information about a Lie algebra than a usual zeta function does. In particular, it is possible for two Lie algebras to have different reduced zeta functions counting ideals, but the same ideal reduced zeta function (see Example~\ref{coinc} below).

Throughout the paper, $\mN$ denotes the set of all positive integers. 
We denote by $\mZ_{\ge 0}$, $\mR_{\ge 0}$, $\mR_{\ge 0}$ the sets of nonnegative 
integers, nonnegative real numbers, positive real numbers respectively. 
We write $S[\![t]\!]$ for the ring of formal power series over a commutative ring $S$.

\textbf{Acknowledgements.} The author is very grateful to Marcus du Sautoy for his continuous support. Luke Woodward helped the research enormously by providing many explicit calculations of local zeta functions. The author would also like to thank Fran\c{c}ois Loeser, Misha Gavrilovich, Jeffrey Gianciracusa, and Gareth Jones for useful discussions and advice and both referees for spotting a number of errors and for helpful suggestions. The author was supported by a Scatcherd European Scholarship.

\section{Motivic zeta functions and the Euler characteristic} \label{motivic}

In order to be able to define a reduced zeta function for all Lie rings $L$ (not just for those whose corresponding zeta function is uniform) and to analyse it rigorously, we need the concept of a motivic zeta function developed by du Sautoy and Loeser~\cite{duSL}. We recall the necessary definitions and results from~\cite{duSL}.

Throughout the remainder of this section, $k$ will denote a field of characteristic zero. The \emph{Grothendieck ring} $\mathcal{M}$ of algebraic varieties over $k$ is generated by symbols $[S]$, for each algebraic variety $S$ over $k$, with the relations

\begin{enumerate}[(i)]
\item $[S]=[S']$ if $S$ is isomorphic to $S'$;
\item $[S]=[S\setminus S']+[S']$ if $S'$ is closed in $S$;
\item $[S\times S']=[S][S']$. 
\end{enumerate} 
Set $\mathbf{L}:=[\mathbb{A}_k^1]$ in $\mathcal{M}$, where $[\mathbb{A}_k^r]$ is the affine $r$-dimensional space over $k$.
Denote by $\mathcal{M}_{\loc}$ the ring $\mathcal{M}[\mathbf{L}^{-1}]$ obtained by localisation, and let $\mathcal{M}[T]_{\loc}$ be the subring of $\mathcal{M}_{\loc}[\![T]\!]$ generated by $\mathcal{M}_{\loc}[T]$ and all the series $(1-\mathbf{L}^a T^b)^{-1}$ where $a\in\mZ$ and $b\in\mathbb{N}$.  

Let $Y$ be an algebraic variety over $k$, that is a reduced separated $k$-scheme of finite type.  If $K$ is a field which is a finite extension of $k$, write $Y(K)$ for the set of rational points of $Y$ over $K$.  By the set underlying $Y$ we shall mean the set of closed points of $Y$.
This set is in one-to-one correspondence with the set consisting, for each finite extension $K$ of $k$, of the orbits of rational points of $Y$ over $K$ under the action of the Galois group $\Gal(K/k)$. A subset of $Y$ is said to be \emph{constructible} if it can be obtained from Zariski closed subsets of $Y$ by taking finite unions and complements. If $A$ is a constructible subset and $K$ is a finite extension of $k$, one can define the set $A(K)$ of points of $A$ over $k$ as follows:
\[
A(K):= \{ x\in Y(K): \text{ the orbit of } x \text{ under } \Gal(K/k) \text{ belongs to } A \}.
\]
To any constructible subset $A$ one can associate a canonical element $[A]$ of the Grothendieck ring $\Cal{M}$. 

 Let $L$ be a Lie algebra over $k[\![t]\!]$, finitely generated as a $k[\![t]\!]$-module. By a class $\calX$ of subalgebras of $L$ we understand the data consisting, for every field $K$ which is a finite extension of $k$, of a family $\calX(K)$ of subalgebras of $L\otimes_k K$. We shall be concerned only with the classes $\calX^{\le}$ and $\calX^{\vtl}$ where $\calX^{\le}(K)$ is the set of closed subalgebras of $L\otimes_k K$ which are $K[\![t]\!]$-submodules of $L\otimes_k K$ and $\calX^{\vtl}(K)$ is the set of closed ideals of $L\otimes_{k} K$ that are $K[\![t]\!]$-ideals of $L\otimes_k K$.  In what follows, we shall assume that $\calX$ is one of these two classes.

We shall view the finite Grassmannian $X_n=\Gr(L/t^n L)$ as an algebraic variety  over $k$. The points of $X_n$ over a finite extension $K$ of $k$ are vector 
subspaces of $(L\otimes_{k} K)/(t^n L \otimes_k K)$. Define
$$
A_n(\calX)(K):=\{H\in \calX(K):\codim_{L\otimes_k K} H = n \}. 
$$
Note that $t^n L\otimes_k K$ is contained in every element of $A_n(\calX)(K)$. Thus, $A_n(\calX)(K)$ can be thought of as a subset of $X_n(K)$. By~\cite{duSL}, Lemma 2.12, there exists a constructible subset $B_n$ of $X_n$ such that 
$B_n(K)=A_n (\calX)(K)$ for all finite extensions $K$ of the field $k$. Clearly, such $B_n$ is unique, so we can associate to $A_n (\calX)$ the element $[B_n]$ of the Grothendieck ring. The \emph{motivic zeta function} of $L$ and $\calX$ is defined by
\[ P_{L,\calX}(T)=\sum_{n=0}^{\infty} [B_n] T^n. \]
It takes values in $\calM[\![T]\!]$. The following result establishes the rationality of a motivic zeta function.

\begin{thm}\label{rational} \emph{(\cite{duSL},Theorem 5.1 (1))} Let $k$ be a field of characteristic zero. Let $L$, $\calX^{\le}$, $\calX^{\vtl}$ be as above. Then $P_{L,\calX^{*}}(T)$ belongs to $\calM[T]_{\loc}$
for $*\in \{\le,\vtl\}$.
\end{thm}

\begin{remark} It is stated in~\cite{duSL}, Subsection 2.8, that closed points of an algebraic variety $Y$ correspond to rational points of $Y$ over finite extensions $K$ of $k$. However, in fact closed points of $Y$ correspond to \emph{Galois orbits} on $Y(K)$, as stated above. On account of this, the definition of a constructible class of subalgebras in~\cite{duSL}, Subsection 2.9, should be amended as follows: $\calX$ is a~\emph{constructible class of subalgebras} if there exists a constructible subset $B_{l,n}$ of $X_l$ such that $B_{l,n}(K)=A_{l,n}(\mathcal{X})(K)$ for all finite extensions $K$ of $k$ (here we use the notation of~\cite{duSL}). The author is grateful to one of the referees for pointing out the same error in the original version of the present paper.  
\end{remark}

Now assume $k=\mQ$. For any variety $X$ over $\mQ$, one can choose a model $X'$ of $X$ over $\mZ$ and consider the number of points $n_p(X)$ of the reduction $X'$ modulo $p$, where $p$ is a prime number. Then $n_p(X)$ is well defined (that is, does not depend on the choice of $X'$) for all but finitely many $p$. Let $\Cal P$ be the set of all primes. If $S$ is a ring, denote by $S^{\Cal P'}$ the ring $\prod_{p\in \Cal P} S/ \oplus_{p\in \Cal P} S$. 
(Here, $\prod_{p\in P} S$ is the ring of functions $\Cal P \ra S$, and $\oplus_{p\in \Cal P} S$ 
is the ideal consisting of such functions with finite support.) Then the sequence 
$(n_p)_{p\in \Cal P}$ induces a ring homomorphism $n: \Cal M \ra \mZ^{\Cal P'}$ (see~\cite{duSL}, Subsection 7.6, for more detail).

Setting $n_p(\bL^{-1})=1/p$, 
one can extend $n$ to a map $n: \Cal M_{\loc}[T] \ra \mQ [T]^{\Cal P'}$. 
(To evaluate $n$ at a polynomial with coefficients from $\Cal M_{\loc}$, apply it term-by-term.) 
Sending $(1-\bL^a T^b)^{-1}$ to $(1-p^a T^b)^{-1}$, we may further extend 
$n$ to a ring homomorphism $n: \Cal M[T]_{\loc} \ra \mQ [\![T]\!]^{\Cal P'}$. 
Now let 
$\tilde{n}: \Cal M[T]_{\loc} \ra \prod_{p\in \Cal P} Q[\![T]\!]$ be a lifting of the map $n$ (that is, a set-theoretic map that yields $n$ when composed with the canonical projection), and let $n_p$ be $\tilde{n}$
composed with the projection onto the $p$-component. Then $n_p$ is not determined uniquely, but if $(n'_p)_{p\in \Cal P}$ is another sequence of maps obtained in this way, then for each (fixed) $x\in \Cal M[T]_{\loc}$, 
$n_p (x)=n'_p (x)$ for almost all $p$. 

If $L$ is a Lie algebra of finite rank defined over $\mZ$ then the local zeta functions of $L$ are closely related to the motivic zeta function of $L\otimes_{\mZ} \mQ[\![t]\!]$. The following is effectively a restatement of 
\cite{duSL}, Corollary 7.9, and its analogue for ideal zeta functions.
\begin{thm} \label{mot-local} Let $L$ be a Lie algebra over $\mZ$, of finite rank as a $\mZ$-module. For almost all $p$,
\[ n_p(P_{L\otimes\mQ[\![t]\!],\calX^{*}}(T))=\tzeta_{L\otimes \mZ_p}^{*}(T). \]
\end{thm}

In addition to motivic zeta functions, we shall need the notion of Euler characteristic for constructible sets over $\mC$ (that is, for constructible subsets of $\mC^m$, $m\in\mZ_{\ge 0}$). Such a concept is defined (in greater generality) by van den Dries~\cite{vdDries}. 
Say that a map $f:A\ra B$ between two constructible sets $A$ and $B$ over $\mC$ is 
\emph{constructible} if
the graph $\{(x,f(x)): x\in A\}$ is constructible.
The next result follows from 
 \cite{vdDries}, Chapter 4, \S 2 (particularly, Section 2.15).

\begin{thm}\label{Euler} There exists a unique integer-valued function $\chi$ on the class of all constructible sets over $\mC$ satisfying the following properties:
\begin{enumerate}[(i)]
\item $\chi(\mC^m)=1$ for all $m$,
\item $\chi(A\cup B)=\chi(A)+\chi(B)$ for disjoint constructible 
$A,B\subseteq \mC^m$,
\item If $f:B\ra A$ is a constructible map between constructible sets $A\subseteq \mC^m$ and $B\subseteq \mC^n$, and $e\in \mN$ is such that $\chi(f^{-1}(\{a\}))=e$ for all $a\in A$, then
$$\chi(B)=e\cdot \chi(A).$$
\end{enumerate}
\end{thm}

\begin{remark} The function $\chi$ coincides with the topological Euler-Poincar\'e characteristics.
\end{remark}

The Euler characteristic $\chi(S)$ of a constructible set $S$ is therefore an invariant of the element $[S]$ of the Grothendieck ring $\Cal{M}$. It induces a ring homomorphism $\chi:\Cal{M}\ra \mZ$. We can extend $\chi$ to a ring homomorphism $\chi:\Cal{M}_{\loc}\ra \mZ$ setting $\chi(\bL^{-1})=1$. Further, $\chi$ extends to a homomorphism 
$$\chi: \Cal{M}_{\loc}[\![T]\!] \ra \mZ [\![T]\!] $$
by the obvious rule
$$
\chi \left( \sum_{j=0}^{\infty} A_j T^j \right) = 
\sum_{j=0}^{\infty} \chi(A_j) T^j.
$$  

\section{Reduced zeta functions} \label{main}

From now on we shall mostly work over the field $\mC$ of complex numbers. Let $L$ be a
Lie algebra over $\mC[\![t]\!]$, finitely generated as a $\mC[\![t]\!]$-module.
 We define the \emph{reduced zeta functions} of $L$ by
$$
 R_L^{*}(T)=\chi(P_{L,\Cal{X}^*}(T)), 
$$
where $*$ is either $\le$ or $\vtl$ as usual. We shall refer to $R_L^{\le}(T)$ and $R_L^{\vtl}(T)$ as the \emph{subalgebra} and the \emph{ideal} reduced zeta functions of $L$ respectively. If $L$ is a Lie algebra over a subring $S$ of $\mC[\![t]\!]$, finitely generated as an $S$-module, define the reduced zeta functions of $L$ by 
$$R_L^{*}(T)=R_{L\otimes_S \mC[\![t]\!]}^{*}(T).$$

\begin{remark} This definition has much in common with the concept of topological zeta functions defined by du Sautoy and Loeser in~\cite{duSL}, Section 8. The difference is that in the present approach $T$ is considered as an independent variable and is not replaced by $\bL^{-s}$. Reduced and topological zeta functions encode different kinds of information about a Lie algebra. Apparently, neither of these functions can be derived from the other. Unlike the reduced zeta function (see Theorem~\ref{mult} below), the topological zeta function is not multiplicative with respect to taking direct sums in any generality.
\end{remark}

\begin{prop} The reduced zeta functions $R_L^*(T)$ are rational functions in $T$.
\end{prop}
\begin{proof} This follows immediately from Theorem \ref{rational}. \end{proof}

If $L$ is a Lie algebra over $\mZ$ of finite rank and its motivic zeta function is of the form  $P_{L\otimes \mQ[\![t]\!],\Cal{X}^*}(T)=f(\bL,T)$ where $f\in \mZ (X,Y)$, then by Theorem \ref{mot-local} we have  $\tzeta_{L\otimes \mZ_p}(T)=f(p,T)$ for almost all primes $p$. In this case 
\begin{equation} \label{connection}
R_L^{*}(T)=f(1,T), 
\end{equation}
 so the reduced zeta function is obtained by replacing $p$ with $1$ in the expression for $\tzeta_{L\otimes \mZ_p}(T)$, as indicated in 
Section~\ref{reducedover}. 

We now analyse ways to simplify calculation of a reduced zeta function. Assume $L$ is torsion-free. Let $d$ be the dimension of $L$ as a $\mC[\![t]\!]$-module. Fix a $\mC[\![t]\!]$-basis $\{x_1,\ldots,x_d\}$ of $L$.  Using this basis, we shall write elements of $L$ simply as row vectors with entries from $\mC[\![t]\!]$. If $z\in L$, write $f_i(z)$ for the $i$-th coordinate of $z$ in this basis. If $\mbf{n}=(n_1,\ldots,n_d)$ is a $d$-tuple of nonnegative integers, let $\Cal{T}_{\mbf{n}}$ be the set of matrices of the form
\begin{equation} M=\left(\begin{array}{cccc}
 t^{n_1} & a_{12} &\ldots & a_{1d} \\
0 & t^{n_2} & \ldots & a_{2d} \\
\vdots & \vdots & \ddots & \vdots \\ 
0 & 0 & \ldots & t^{n_d} 
\end{array}\right)    \label{normalform} 
\end{equation}
where $a_{ij}\in \mC[\![t]\!]$ is a polynomial in $t$ of degree less than $n_j$  (by convention, the zero polynomial has negative degree). We shall write
\[ a_{ij}(t)=a_{ij}^{(0)}+a_{ij}^{(1)}t+\cdots + a_{ij}^{(n_j-1)} t^{n_j-1}, \]
where $a_{ij}^{(l)}\in \mC$. 
Clearly, $\Cal{T}_{\mbf{n}}$ may be viewed as the set underlying the affine space of dimension $n_2+2n_3+3n_4+\cdots+(d-1)n_d$ over $\mC$. If $M\in \Cal{T}_{\mbf{n}}$, let $\Xi(M)$ be the $\mC[\![t]\!]$-submodule of $L$ spanned by the rows of $M$. 

If $m$ is a nonnegative integer, let $\Cal{N}_m$ be the set of all $d$-tuples $\mbf{n}=(n_1,\ldots,n_d)$ of nonnegative integers such that $n_1+\cdots+n_d=m$. If $H$ is a $\mC[\![t]\!]$-submodule of $\mC$-codimension $m$ in $L$, there is a unique matrix $M\in\bigcup_{\mbf{n}\in \Cal{N}_m} \Cal{T}_{\mbf{n}}$ such that $\Xi(M)=H$. Indeed, every $d\times d$-matrix over $\mC[\![t]\!]$ with a non-zero determinant can be transformed by elementary row operations to a unique matrix of the form \eqref{normalform}. (Here, the elementary row operations are: permuting rows of $M$, adding a row of $M$ to another row
of $M$ with a coefficient from $\mC [\![t]\!]$ and multiplying a row of $M$ by an invertible element of
$\mC [\![t]\!]$.)

Let $v$ be the standard valuation on $\mC[\![t]\!]$; that is, for $a\ne 0$, $v(a)=r$ where $r$ is the greatest integer such that $t^r$ divides $a$. Consider a matrix $M\in \Cal{T}_{\mbf{n}}$ (for some $\mbf{n}$) and an element $z\in L$ written as a row vector as indicated above. Let $m_j$ be the $j$-th row of $M$. Define
recursively for $i=1,\ldots,d$ a condition $\Cal{D}_i (z,M)$ and an element 
$\tau_i (z,M)\in \mC[\![t]\!]$ as follows:
$$
\Cal{D}_i (z,M) \text{ holds if and only if } \; 
v \left( f_i (z)-\sum_{j=1}^{i-1} \tau_j (z,M) f_i(m_j) \right)\ge n_i,
$$
$$\tau_i (z,M)=f_i \left(z-\sum_{j=1}^{i-1} \tau_j (z,M)m_j \right)t^{-n_i}.$$
Note that $\tau_i (z,M)$ is well defined if $\Cal{D}_1 (z,M),\ldots,\Cal{D}_i (z,M)$ hold. Say that $\Cal{D} (z,M)$ holds if and only if $\Cal{D}_1 (z,M),\ldots,\Cal{D}_d (z,M)$ all hold.
It is easy to see that $z$ is in the $\mC[\![t]\!]$-span of the 
rows of $M$ if and only if $\Cal{D} (z,M)$ holds.
 Hence $\Xi(M)$ is an ideal of $L$ if and only if $\Cal{D} ([m_i,x_j],M)$ holds for all  
$i,j \in \{1,\ldots,d\}$; and $\Xi(M)$ is a subalgebra of $L$ if and only if $\Cal{D}([m_i,m_j],M)$ holds for all $i,j \in \{1,\ldots,d\}$.
 
Therefore, for each choice of $\mbf{n}$, the condition $\Xi(M)\in \Cal{X}^*$ may be written as a finite Boolean combination of polynomial equations in the coordinates $a_{ij}^{(r)}$. That is, there is a constructible subset $\Cal{F}_{\mbf{n}}^*$ of $\Cal{T}_{\mbf{n}}$ whose elements are precisely the matrices $M$ such that $\Xi(M)$ is in the class $\calX^*(\mC)$. Then
\begin{equation}
 R_L^*(T)=\sum_{\mbf{n}\in \mZ_{\ge 0}^{d}} \chi(\Cal{F}_{\mbf{n}}^*)
T^{n_1+\cdots+n_d}. \label{matform}
\end{equation}   

Call a pair $(i,j)$ (where $1\le i<j \le d$) \emph{removable} if there exist integers $l_1,l_2,\ldots,l_d$ such that 
\begin{enumerate}[(i)]
\item For all $z\in \mC^*= \mC \setminus \{0\}$, the map given by $x_r\mapsto z^{l_r}x_r$ ($r=1,\ldots,d$) is an automorphism of $L$, 
\item $l_i\ne l_j$.  
\end{enumerate}
If $C$ is a set of pairs $(i,j)$  (satisfying $1\le i<j \le d$), let $\Cal{E}_{\mbf{n},C}$ be the subvariety of $\Cal{T}_{\mbf{n}}$ given by the following condition:
$ a_{ij}=0$ whenever $ (i,j)\in C. $

\begin{thm} \label{remove} 
Let $L$ be a torsion-free Lie algebra over $\mC[\![t]\!]$. Let 
$\{x_1,\ldots,x_d\}$ be a basis of $L$. 
Suppose $C$ is a set of removable pairs $(i,j)$. Then 
$$
R_L^*(T)=\sum_{\mbf{n}\in\mZ_{\ge 0}^{d}} 
\chi(\Cal{E}_{\mbf{n},C}\cap\Cal{F}_{\mbf{n}}^*) T^{n_1+\cdots+n_d}. 
$$
That is, one may assume that $a_{ij}=0$ for each removable pair $(i,j)$ when calculating the reduced zeta function. 
\end{thm}

\begin{proof} We use induction on the size of $C$. The base case $C=\varnothing$ follows from \eqref{matform}.
Assuming $C$ is non-empty, pick a pair $(i,j)\in C$,
 and let $D=C\setminus \{(i,j)\}$. By the inductive hypothesis,
\begin{equation}
R_L^*(T)=\sum_{\mbf{n}\in\mZ_{\ge 0}^{d}} 
\chi(\Cal{E}_{\mbf{n},D}\cap\Cal{F}_{\mbf{n}}^*) T^{n_1+\cdots+n_d} \label{indhyp} 
\end{equation}
Let $\mbf{n}$ be a $d$-tuple of integers.
If $w$ is an integer satisfying $0\le w<n_j$, let $\Cal{S}_{\mbf{n},i,j,w}$ be the subvariety of $\Cal{T}_{\mbf{n}}$ defined by the condition $v(a_{ij})=w$ (that is, $a_{ij}^{(r)}=0$ for all $r<w$ and $a_{ij}^{(w)}\ne 0)$. By the additivity of Euler characteristic on constructible sets,
\[ \chi(\Cal{E}_{\mbf{n},D}\cap\Cal{F}_{\mbf{n}}^*)=\chi(\Cal{E}_{\mbf{n},C}\cap\Cal{F}_{\mbf{n}}^*)+
\sum_{w=0}^{n_j-1} \chi(\Cal{S}_{\mbf{n},i,j,w}\cap \Cal{E}_{\mbf{n},D}\cap\Cal{F}_{\mbf{n}}^*). \]   
Thus it is enough to prove that 
$\chi(\Cal{S}_{\mbf{n},i,j,w}\cap \Cal{E}_{\mbf{n},D}\cap\Cal{F}_{\mbf{n}}^*)=0$ whenever $0\le w<n_j$. Indeed, then the result would follow from \eqref{indhyp}.

Let  
$ f:\Cal{S}_{\mbf{n},i,j,w}\cap \Cal{E}_{\mbf{n},D}\cap\Cal{F}_{\mbf{n}}^* \ra \mC^* $
be the map defined by
$ M \mapsto a_{ij}^{(w)}. $
 Since $(i,j)$ is a removable pair, there are integers
 $l_1,\ldots,l_d$ such that $l_i\ne l_j$ and, for all $z\in\mC^*$, $x_r\mapsto z^{l_r}x_r$ ($r=1,\ldots,d$) gives rise to an automorphism $\phi_z$ of $L$. If
$M\in \Cal{S}_{\mbf{n},i,j,w}\cap \Cal{E}_{\mbf{n},D}$ and $z\in\mC^*$,  
let 
\begin{equation*} g_z(M)=\left(\begin{array}{cccc}
 t^{n_1} & z^{l_2-l_1} a_{12} & \ldots & z^{l_d-l_1} a_{1d} \\
0 & t^{n_2} & \ldots & z^{l_d-l_2} a_{2d} \\
\vdots & \vdots & \ddots & \vdots \\ 
0 & 0  & \ldots & t^{n_d} 
\end{array}\right), 
\end{equation*}
so $g_z(M)$ is obtained from $M$ by multiplying the $r$-th column by $z^{l_r}$ and the 
$r$-th row by $z^{-l_r}$, for $r=1,2,\ldots,d$. Then 
$ \Xi(g_z(M))=\phi_z(\Xi(M)). $
Since ideals and subalgebras are preserved by automorphisms of $L$, we conclude that $g_z$ is an automorphism of 
$\Cal{S}_{\mbf{n},i,j,w}\cap \Cal{E}_{\mbf{n},D}\cap\Cal{F}_{\mbf{n}}^*$. 
Note that $f(g_z(M))=z^{l_j-l_i}f(M)$. Thus, for any $a\in \mC^{*}$,
 $g_z |_{f^{-1}(\{a\})}$ is a constructible bijection onto 
$f^{-1}(\{z^{l_j-l_i}a\})$. Since $z^{l_j-l_i}$ runs through all non-zero complex numbers as $z$ runs through $\mC^{*}$, by Theorem \ref{Euler}, 
$\chi(f^{-1}(\{a\}))$ is a constant independent of $a\in \mC^*$, $e$ say. 
By Theorem \ref{Euler} again, 
$$
\chi(\Cal{S}_{\mbf{n},i,j,w}\cap \Cal{E}_{\mbf{n},D}\cap\Cal{F}_{\mbf{n}}^*)=
 e\cdot\chi(\mC^*)=0
$$
because $\chi(\mC^*)=0$. This completes the proof of the theorem. \end{proof}   

Unlike usual zeta functions, reduced zeta functions are multiplicative with respect to direct sums in many cases. 

\begin{thm} \label{mult} Let $L$ and $N$ be finitely generated (as $\mC[\![t]\!]$-modules) torsion-free Lie algebras over $\mC[\![t]\!]$. Then
$$R_{L\oplus N}^{\vtl} (T) = R_L^{\vtl}(T) R_N^{\vtl} (T).$$
Suppose further that there is a basis $\{x_1,\ldots,x_d\}$ of $L$ such that, 
for all $j\in [1,d]$, there exist integers $l_{j1},\ldots,l_{jd}$ with the property that
 $l_{jj}\ne 0$ and, for all $z\in \mC^*$, the map $x_r\mapsto z^{l_{jr}} x_r$ 
induces an automorphism of $L$. Then 
$$R_{L\oplus N}^{\le} (T) = R_L^{\le}(T) R_N^{\le} (T).$$
\end{thm}

\begin{remark} The condition on the basis of $L$, while true for many examples considered in the literature, is quite restrictive: loosely speaking, it fails for a `random' Lie algebra. It is an interesting question whether the subalgebra reduced zeta function is multiplicative in general.
\end{remark}

\begin{proof} We shall use notation described above adding subscripts $L$, $N$ and $L\oplus N$ to denote the Lie algebra we refer to. Choose bases $\{x_1,\ldots,x_d\}$ for $L$ and $\{y_1,\ldots,y_e\}$ for $N$. As above, elements of $L$ and $N$ can be represented by row vectors of length $d$ and $e$ respectively. Consider first the case of ideals. By \eqref{matform}, it suffices to prove that
\begin{equation} 
\chi(\Cal{F}^{\vtl}_{(\mbf{n},\mbf{m}),L\oplus N}) 
= \chi(\Cal{F}^{\vtl}_{\mbf{n},L})
\chi(\Cal{F}^{\vtl}_{\mbf{m},N})  \label{prod1}
\end{equation} 
for any $d$-tuple $\mbf{n}$ and any $e$-tuple $\mbf{m}$. 
Fix $\mbf{n}$ and $\mbf{m}$. 
Let $\scr M_{\mbf m}$ be the set of all $d\times e$ matrices $C=(c_{ij})$ with entries from 
$\mC[\![t]\!]$ such that $\deg(c_{ij})<m_j$ for all $i,j$.
Let $A\in \Cal{T}_{\mbf{n}}$ and $B\in \Cal{T}_{\mbf{m}}$. If 
$C\in \scr{M}_{\mbf m}$,
write 
$$M=M(A,B,C)=
\begin{pmatrix} A & C \\ 0 & B \end{pmatrix}.$$
Note that if $M\in \Cal{F}^{\vtl}_{(\mbf{n},\mbf{m}),L\oplus N}$, then $A\in \Cal{F}^{\vtl}_{\mbf{n}}$ and 
$B\in \Cal{F}^{\vtl}_{\mbf{m}}$. Let
$$
\Cal{Y}^{\vtl}_{\mbf{n},\mbf{m}}(A,B)= 
\{C\in \scr{M}_{\mbf m} : 
M(A,B,C) \in \Cal{F}^{\vtl}_{(\mbf{n},\mbf{m}),L\oplus N} \}.
$$
Then, by Theorem \ref{Euler}, $\eqref{prod1}$ holds provided, for all $A\in \Cal{F}^{\vtl}_{\mbf{n}}$, $B\in \Cal{F}^{\vtl}_{\mbf{m}}$,
\begin{equation}
\chi(\Cal{Y}^{\vtl}_{\mbf{n},\mbf{m}}(A,B)) = 1. \label{prod2}
\end{equation} 

Fix  $A\in \Cal{F}^{\vtl}_{\mbf{n}}$ and $B\in \Cal{F}^{\vtl}_{\mbf{m}}$.
Writing $c_i$ for the $i$-th row of $C$, observe that $C\in \Cal{Y}^{\vtl}_{\mbf{n},\mbf{m}}(A,B)$ 
if and only if $[c_i,y_j]_{N}$ is in the span of the rows of $B$ whenever 
$1\le i\le d$, $1\le j\le e$. Thus $\Cal{Y}^{\vtl}_{\mbf{n},\mbf{m}}(A,B)$ is a vector space over $\mC$, so \eqref{prod2} holds.

Finally consider the case of subalgebras. We may assume that the basis 
is as in the hypothesis. Let 
$C=\{(i,j)\in \mN^2: 1\le i\le d<j\le d+e\}$. By the hypothesis, all pairs in $C$ are removable with respect to $L\oplus N$ (to obtain a required automorphism of $L\oplus N$, combine the identity automorphism of $N$ with one of the automorphisms of $L$ from the hypothesis). Hence by Theorem \ref{remove}
\begin{eqnarray*} R^{\le}_{L\oplus N} (T) & = & 
 \sum_{\mbf{n}\in\mZ_{\ge 0}^{d}}\sum_{\mbf{m}\in\mZ_{\ge 0}^{e}} 
\chi(\Cal{E}_{(\mbf{n},\mbf{m}),C,L\oplus N}\cap
\Cal{F}_{(\mbf{n},\mbf{m}),L\oplus N}^{\le})
T^{n_1+\cdots+n_d+m_1+\cdots+m_e}
 \\
& = & \sum_{\mbf{n}\in\mZ_{\ge 0}^{d}}\sum_{\mbf{m}\in\mZ_{\ge 0}^{e}}
\chi(\Cal{F}^{\le}_{\mbf{n},L}\times \Cal{F}^{\le}_{\mbf{m},N})
T^{n_1+\cdots+n_d+m_1+\cdots+m_e} \\
& = & R_L^{\le} (T) R_N^{\le} (T). 
\qquad\qquad\quad\qquad\qquad\qquad\qquad\qquad\qquad\qquad\qquad \qed
\end{eqnarray*}\renewcommand{\qed}{}\end{proof}

\section{Lie algebras with a nice and simple basis} \label{nice}

Let $L$ be a torsion-free $d$-dimensional Lie algebra over $\mC[\![t]\!]$, as above. Let $B=\{x_1,\ldots,x_d\}$ be a basis of $L$.  Let $\Cal{D}_B^*$ be the set of all $d$-tuples $\mbf{n}$ such that the rows of the diagonal matrix  
\begin{equation*} D_{\mbf{n}}=\left(\begin{array}{cccc}
 t^{n_1} & 0  &\ldots & 0 \\
0 & t^{n_2} & \ldots & 0 \\
\vdots & \vdots & \ddots & \vdots \\ 
0 & 0  & \ldots & t^{n_d} 
\end{array}\right)   
\end{equation*}
span an ideal (if $*=\vtl$) or a subalgebra (if $*=\le$). Call a basis $B$ \emph{nice} (with respect to $*$) if 
\begin{equation}
R_L^*(T)=\sum_{\mbf{n}\in\Cal{D}_B^*} T^{n_1+n_2+\cdots+n_d} \label{niceform}
\end{equation}
By Theorem \ref{remove}, $B$ is nice whenever all the pairs $(i,j)$ ($1\le i<j\le d$) are removable. Call $B$ \emph{simple} if for all $i,j$ there exist 
$l\in\{1,\ldots,d\}$ and $a\in \mC[\![t]\!]$ such that $[x_i,x_j]=a x_l$ and either $a=0$ or $v(a)=0$.
If $B$ is simple, consider the polyhedral cones
\begin{eqnarray} 
\Cal{C}_B^{\vtl}& = & \{\mbf{y}\in \mR_{\ge 0}^d: y_l\le y_i\; \text{and}\; y_l\le y_j\; \text{whenever} \;
[x_i,x_j]=ax_l,\; a\ne 0 \}, \label{coneideal} \\
\Cal{C}_B^{\le}& = & \{\mbf{y}\in \mR_{\ge 0}^d: y_l\le y_i+y_j \; \text{whenever} \;
[x_i,x_j]=ax_l,\; a\ne 0 \}.  \label{conesubalgebra}
\end{eqnarray} 

\begin{prop} \label{calc} Let $L$ be a torsion-free $d$-dimensional Lie algebra over 
$\mC[\![t]\!]$. Assume $L$ has a nice (with respect to $*$) and simple basis 
$\{x_1,\ldots,x_d\}$. Let $\Cal{C}_B^*$ be as above. Then
\begin{equation}
R_L^*(T)=\sum_{\mbf{n}\in\Cal{C}_B^*\cap\mZ^d} T^{n_1+n_2+\cdots+n_d}. \label{sumcone}
\end{equation} 
\end{prop}
\begin{proof} Observe that $\Cal{D}_B^*=\Cal{C}_B^*\cap \mZ^d$ for each of the two possible values of $*$. Thus, the result follows from \eqref{niceform}. \end{proof} 

The expression \eqref{sumcone} is a sum over all integer points of a polyhedral cone. Sums of this kind may be computed using the Elliott-MacMahon algorithm (see \cite{Stanley1973}, Section 3). The running time of the algorithm is exponential in the dimension of $L$, so the algorithm is only practical for Lie algebras of relatively small dimensions. 

Proposition~\ref{calc} together with Theorem~\ref{remove} provide a way of calculating reduced zeta functions of Lie algebras possessing a simple basis with respect to which all pairs $(i,j)$ ($1\le i<j \le d$) are removable. We list below some examples of reduced zeta functions found using this method. One can use this procedure to calculate, with relatively little effort, reduced zeta functions of many other Lie algebras. However, it appears that a Lie algebra needs to possess a basis that is simple or is `close' to being simple in order to be amenable to Theorem~\ref{remove}. Thus, the theorem simplifies calculations only for a special class of Lie algebras. 

\begin{exam} Consider
\[ F_{2,n}=\langle x_1,x_2,\ldots,x_n,\; y_{ij} \: (1\le i<j\le n) \;| \; [x_i,x_j]=y_{ij} \; \text{for} \; i<j \rangle, \] 
the free Lie algebra of nilpotency class 2 on $n$ generators. (Here, and in the sequel, the omitted  Lie brackets between elements of the basis are assumed to be $0$).  For every $n$-tuple $\mbf{l}=\{l_1,\ldots,l_n\}$ of integers and for all $z\in\mC^*$, let $\psi_{\mbf{l},z}:F_{2,n} \ra F_{2,n}$ be the linear map given by
\begin{eqnarray*}
 x_i & \mapsto & z^{l_i} x_i \quad \quad\;\;\text{ for all } i, \\  
 y_{ij} & \mapsto & z^{l_i+l_j} y_{ij} \quad \text{ whenever } i<j.   
\end{eqnarray*}
One can check that $\psi_{\mbf{l},z}$ is an automorphism of $F_{2,n}$. The maps 
$\psi_{\mbf{l},z}$, as $\mbf{l}$ varies, witness removability of all pairs 
$(l,s)$ (with $l<s$) with respect to the basis 
$$
B_n=\{x_1,\ldots,x_n,\; y_{ij} \: (1\le i<j\le n)\}.
$$ 
Thus, $B_n$ is a nice and simple basis of $F_{2,n}$. It follows that if $C$ is a subset of $B_n$, then the quotient of $F_{2,n}$ by the smallest ideal containing $C$ has an obvious nice and simple basis. This provides many examples of Lie algebras of class $2$ with a nice and simple basis.  
\end{exam}

\begin{exam}
More specifically, consider the Heisenberg Lie algebra 
$$
H=F_{2,2}=\langle x,y,z \; | \; [x,y]=z \rangle.
$$
By Proposition \ref{calc}, 
\begin{eqnarray*}
R_{H}^{\vtl}(T) & = & \sum_{n_3=0}^{\infty} \sum_{n_1=n_3}^{\infty}\sum_{n_2=n_3}^{\infty} T^{n_1+n_2+n_3}  \\
& = & \sum_{n_3=0}^{\infty} \sum_{m_1=0}^{\infty} \sum_{m_2=0}^{\infty}
T^{3n_3+m_1+m_2} = \frac{1}{(1-T^3)(1-T)^2}.
\end{eqnarray*}
(Here we substitute $m_1=n_1-n_3$, $m_2=n_2-n_3$.) It is also routine to calculate the subalgebra zeta function of $H$:
$$R_H^{\le}(T)=\frac{1+T+T^2}{(1-T^2)^2 (1-T)}.$$
The reduced zeta functions of $H^{n}=\overbrace{H\oplus\cdots \oplus H}^{n}$ can be immediately derived via Theorem \ref{mult}: 
$R_{H^n}^{*}(T)=(R_H^{*}(T))^n$.
These expressions are much shorter than those of the usual zeta functions of $H^n$, which so far have only be calculated for $n\le 4$ in the case of ideals  and for $n\le 2$ in the case of subalgebras (see~\cite{duSW}).
\end{exam}
   
\begin{exam} We calculate the ideal reduced zeta function of $F_{2,3}$. 
By Proposition \ref{calc},
\[
R_{F_{2,3}}^{\vtl}(T)=\sum_{\mbf{n}\in \Cal{D}_B} T^{n_1+n_2+\cdots+n_6}
\]
where
\[\Cal{D}_B=\{\mbf{n}\in\mZ_{\ge 0}^6: n_1\ge n_4,\, n_1\ge n_5,\, n_2 \ge n_4,\, n_2\ge n_6,\,
n_3\ge n_5,\, n_3\ge n_6\}
\] 
($n_4$, $n_5$, $n_6$ correspond to $y_{12}$, $y_{13}$, $y_{23}$ respectively).
One can calculate this sum by splitting it into the following $6$ disjoint cases: 
$n_4\le n_5\le n_6$, $n_4\le n_6<n_5$, $n_5\le n_6<n_4$, $n_5<n_4\le n_6$, $n_6<n_4\le n_5$,
$n_6<n_5<n_4$; in each of the cases, the sum becomes a geometric series. The result is 
\[ R_{F_{2,3}}^{\vtl}(T)=\frac{1+2T^3+2T^5+T^8}{(1-T)^3 (1-T^3) (1-T^5) (1-T^6)}. \]  
\end{exam}

\begin{exam}
Consider nilpotent Lie algebras of maximal class:
\[M_q=\langle y,x_1,\ldots,x_q \; | \; [y,x_i]=x_{i+1}, \; i=1,\ldots,q-1 \rangle. \]  
 The ideal zeta functions of $M_q$ have been calculated for $q\le 4$ by G. Taylor~\cite{Taylor}, and the expressions become increasingly complicated as $q$ increases. By contrast, it is easy to compute the reduced ideal zeta function of $M_q$ for an arbitrary $q$.
For all $r,l\in \mZ$ and $z\in \mC^*$, the linear map given by $y\mapsto z^r y$, $x_i\mapsto z^{ir+l} x_i$ ($i=1,\ldots,q$) is an automorphism  
of $M_n$. Varying $r$ and $l$, one can easily show that the basis $\{y,x_1,\ldots,x_d\}$ is nice.
Hence,
\begin{eqnarray*}
 R_{M_q}^{\vtl}(T)&=&
\sum_{n_1\ge \ldots \ge n_q\ge 0}\; \sum_{n_0\ge n_2} T^{n_0+n_1+\cdots+n_q} \\
& = & \frac{1}{(1-T)^2\prod_{j=3}^{q+1}(1-T^j)}.
\end{eqnarray*}
\end{exam}

\begin{exam}\label{coinc} Finally, consider the Lie algebra
$$Fil_4=\langle z,x_1,x_2,x_3,x_4 \; | \; [z,x_1]=x_2, [z,x_2]=x_3, [z,x_3]=x_4, [x_1,x_2]=x_4 \rangle.$$
The given basis is nice and simple. The corresponding polyhedral cone   
$$\Cal{C}_{Fil_4}^{\vtl}=\{(y_0,y_1,y_2,y_3,y_4)\in \mR_{\ge 0}^{5} :
y_0\ge y_2 \ge y_3\ge y_4, y_1 \ge y_2\}$$
is the same as the cone \eqref{coneideal} of the maximal class Lie algebra 
$M_4$. Hence
$$R_{Fil_4}^{\vtl}(T)=R_{M_4}^{\vtl}(T)=
\frac{1}{(1-T)^2 (1-T^3) (1-T^4) (1-T^5)}.$$
Note that the usual ideal zeta functions of $Fil_4$ and $M_4$ are not equal (see~\cite{duSW}). 
\end{exam}

\section{Functional equations}\label{funeqs}

It is well known that many uniform zeta functions of Lie algebras (defined over $\mZ$) satisfy a functional equation of the form
\begin{equation}
\tzeta_{L,p}^*(T^{-1})|_{p\ra p^{-1}}=(-1)^{\epsilon}p^a T^b \tzeta_{L,p}^*(T) \label{funeq}
\end{equation}
where $p\ra p^{-1}$ indicates that all occurrences of $p$ in an expression should be replaced by $p^{-1}$ and $\epsilon$, $a$, $b$ are suitable integers (see~\cite{survey}, Subsection 5.7, for more detail). Recently, C.~Voll~\cite{Voll} has proved that a functional equation exists for all ideal zeta functions of Lie algebras of nilpotency class at most $2$ and for all subalgebra zeta functions.

Let $L$ be a Lie algebra over $\mZ$ such that 
$P_{L\otimes_{\mZ} \mQ[\![t]\!],\Cal{X}^*}(T)=f(\bL,T)$
 where $f(X,T)$ is a rational function. Then, by \eqref{connection}, 
$R_{L}^*(T)=f(1,T)$.
If $\tzeta_{L,p}^*$ has a functional equation \eqref{funeq}, so does the reduced zeta function:
\[ R_{L}^*(T^{-1})=(-1)^{\epsilon} T^b R_{L}^*(T). \] 
Thus, functional equations of reduced zeta functions are related to those of usual zeta functions.

In the sequel, we shall investigate functional equations of reduced zeta functions of Lie algebras possessing a nice and simple basis. Let $L$ be a torsion-free $d$-dimensional Lie algebra over $\mC[\![t]\!]$. Assume $L$ has a nice (with respect to a particular choice of $*$) and simple basis $B=\{x_1,\ldots,x_d\}$. 
We shall make use of the following result, which is a restatement of a theorem due to R. Stanley (\cite{Stanley1973}, Theorem 4.1).

\begin{thm}\label{Stanthm} Let $Q_1,Q_2,\ldots,Q_w$ be non-zero homogeneous linear forms with integer coefficients in real variables $z_1,\ldots,z_d$. Let
\begin{eqnarray*}
\Cal{C} &=& \{(z_1,\ldots,z_d)\in \mR_{\ge 0}^d: Q_i (z_1,\ldots,z_d)\ge 0 \quad (i=1,\ldots,w)\}, 
\textrm{ and let } \\
 \Cal{C}^{\mathrm o} &=& \{(z_1,\ldots,z_d)\in \mR_{>0}^d: Q_i (z_1,\ldots,z_d)> 0 \quad (i=1,\ldots,w)\} 
\end{eqnarray*}
be the interior of $\Cal{C}$. Assume $\Cal{C}^{\mathrm o}\cap \mZ^d\ne \varnothing$.
Define
\begin{eqnarray*}
F(X_1,\ldots,X_d)& =& \sum_{\mbf{n}\in \Cal{C}\cap\mZ^d} X_1^{n_1} X_2^{n_2}\cdots X_d^{n_d}\quad
 \textrm{ and } \\  
\bar{F} (X_1,\ldots,X_d)&=& \sum_{\mbf{n}\in \Cal{C}^{\mathrm o}\cap\mZ^d} X_1^{n_1} X_2^{n_2}\cdots X_d^{n_d}.
\end{eqnarray*} Then $F(\mbf{X})$ and $\bar{F}(\mbf{X})$ are rational functions of the $X_i$'s related by\[ \bar{F}(X_1,X_2,\ldots,X_d)=(-1)^{d} F(X_1^{-1},\ldots,X_d^{-1}). \]
\end{thm}
Using Proposition \ref{calc}, we deduce that, whenever $(\Cal{C}_B^*)^{\mathrm o}\cap\mZ^d \ne \varnothing$,
\begin{equation}
R_L^*(T^{-1})=(-1)^d \sum_{\mbf{n}\in(\Cal{C}_B^*)^{\mathrm o}\cap\mZ^d} T^{n_1+\cdots+n_d}.   \label{interior}
\end{equation}

We consider subalgebra zeta functions first. 
\begin{prop} \label{subfe} Suppose $L$, a torsion-free $d$-dimensional Lie algebra over $\mC[\![t]\!]$, has a nice (with respect to $\le$) and simple basis.
 Then 
\[ R_{L}^{\le}(T^{-1})=(-1)^{d} T^d R_{L}^{\le} (T). \]
\end{prop} 
\begin{proof} Certainly, $(\Cal{C}_B^{\le})^{\mathrm o}\cap\mZ^d \ne \varnothing$: 
for example, $(1,1,\ldots,1)\in (\Cal{C}_B^{\le})^{\mathrm o}$. 
Note that the map 
\begin{equation}
 \mbf{n}=(n_1,\ldots,n_d)\mapsto (n_1+1,\ldots,n_d+1) \label{corr}
\end{equation}
is a bijection of $\Cal{C}_B^{\le}\cap\mZ^d$ onto 
$(\Cal{C}_B^{\le})^{\mathrm o} \cap\mZ^d$. Hence the right-hand side of
 \eqref{interior} is equal to $(-1)^{d} T^d R_L^{\le}(T)$, 
and the result follows. \end{proof}

Let us analyse the ideal zeta function of $L$. We assume that $L$ is nilpotent.
Let 
\[ 0=Z_0\le Z_1\le Z_2\le\cdots\le Z_c=L \]
 be the upper central series of $L$   
(so $Z_1=Z(L)$ is the centre of $L$ and $Z_{i+1}/Z_i=Z(L/Z_i)$).
For every element $x\in L$, define the \emph{height} $h(x)$ of $x$ to be the greatest number $r$ such that $x\in Z_r$.
Write  $x_l \prec x_i$ if there exists $j$ such that $[x_i,x_j]=a x_l$, $a\ne 0$. Let $\prec$ be the strict partial order on the set $\{x_1,\ldots,x_n\}$ defined as the transitive closure of the relation specified by the preceding sentence.
 Then \eqref{coneideal} may be rewritten as
\begin{equation} \Cal{C}_B^{\vtl} =  \{\mbf{y}\in \mR_{\ge 0}^d: y_l\le y_i
\; \text{whenever} \;  x_l \prec x_i \}
\end{equation}   
 Note that $h(x_i)>h(x_l)$ whenever $x_l \prec x_i$. Hence $(h(x_1),\ldots,h(x_n))\in \Cal (C_B^{\vtl})^{\mathrm o}$, so 
 $\Cal (C_B^{\vtl})^{\mathrm o} \cap\mZ^n \ne \varnothing$ and we can apply Theorem~\ref{Stanthm}. 

It is easy to understand why the function $R_L^{\vtl}(T)$ does not always have a functional equation: in general, there is no correspondence similar to \eqref{corr}. For example, the ideal reduced zeta function of
\[ L_{W}=\langle z,w_1,w_2,x_1,x_2,y\; |\; [z,w_1]=x_1,\, [z,w_2]=x_2,\, [z,x_1]=y \rangle \]
is 
\[ R_{L_W}^{\vtl}=\frac{1+T+T^2+2T^3+2T^4+2T^5+2T^6+T^7+T^8}
{(1-T)^2 (1-T^2) (1-T^3) (1-T^4) (1-T^5)}. 
\]
(Of course, the given basis is nice and simple.) The Lie algebra $L_W$ is one of the examples of nilpotent Lie algebras whose usual (local) ideal zeta function does not have a functional equation discovered by L.\ Woodward (see~\cite{duSW}).
The following result gives sufficient conditions for a functional equation to hold.

\begin{prop} \label{funeqid} Let $L$ be a torsion-free finite-dimensional nilpotent Lie algebra over $\mC[\![t]\!]$. Suppose $L$ has a nice (with respect to $\vtl$) and simple basis $\{x_1,\ldots,x_d\}$ such that whenever $x_l \prec x_i$ and 
$h(x_i)>h(x_l)+1$, there exists $r$ such that \\ $x_l\prec x_r \prec x_i$. Let $w=\sum_{i=1}^d h(x_i)$. Then
\[ R_{L}^{\vtl}(T^{-1})=(-1)^{d} T^w R_{L}^{\vtl} (T). \]
\end{prop}

\begin{proof} The hypothesis ensures that
$$
\Cal{C}_B^{\vtl} =  \{\mbf{y}\in \mR_{\ge 0}^d: y_l\le y_i
\; \text{whenever} \;  x_l \prec x_i \textrm{ and }
h(x_i)=h(x_l)+1  \}.
$$
Then one can easily check that 
$$\xi:(n_1,\ldots,n_d)\mapsto (n_1+h(x_1),\ldots,n_d+h(x_d))$$
 is a bijection from $\Cal{C}_B^{\vtl}\cap \mZ^d$ onto 
$(\Cal{C}_B^{\vtl})^{\mathrm o}\cap \mZ^d$. The sum of entries of $\xi(\mbf{n})$ exceeds the sum of entries of $\mbf{n}$ by $w$. The result follows by applying \eqref{interior}. \end{proof}

The hypothesis of Proposition \ref{funeqid} is clearly satisfied by any Lie algebra of nilpotency class $2$ with a nice and simple basis. Hence, ideal reduced zeta functions of such algebras have functional equations.

\begin{cor}\label{idefe} Let $L$ be a torsion-free $d$-dimensional Lie algebra over $\mC[\![t]\!]$ of nilpotency class $2$. If $L$ has a nice (with respect to $\vtl$) and simple basis, then
\[  R_{L}^{\vtl}(T^{-1})=(-1)^{d} T^{d+u} R_{L}^{\vtl} (T) \]
where $u=\dim(L/Z(L))$ and $Z(L)$ is the centre of $L$. 
\end{cor}

\providecommand{\bysame}{\leavevmode\hbox to3em{\hrulefill}\thinspace}
\providecommand{\MR}{\relax\ifhmode\unskip\space\fi MR }
\providecommand{\MRhref}[2]{%
  \href{http://www.ams.org/mathscinet-getitem?mr=#1}{#2}
}
\providecommand{\href}[2]{#2}

\end{document}